\documentclass[12pt,reqno]{amsart}
\usepackage{latexsym,amssymb,amscd}
\usepackage{graphicx}
\def\NZQ{\Bbb}               
\def\NN{{\NZQ N}}

\def\kk{{\Bbbk}}
\def\B'c{{\mathcal{B'}}}
\def\U'c{{\mathcal{U'}}}

\def\xb{{\bold x}}

\def\opn#1#2{\def#1{\operatorname{#2}}} 
\opn\chara{char}
\opn\length{\ell}
\opn\projdim{proj\,dim}
\opn\injdim{inj\,dim}
\opn\ini{in}
\opn\rank{rank}
\opn\height{ht}
\opn\Tiefe{Tiefe}
\opn\grade{grade}
\opn\height{ht}
\opn\embdim{emb\,dim}
\opn\codim{codim}

\opn\Tr{Tr}
\opn\bigrank{big\,rank}
\opn\superheight{superheight}\opn\lcm{lcm}
\opn\trdeg{tr\,deg}%
\opn\reg{reg}
\opn\lreg{lreg}
\opn\deg{deg}
\opn\lcm{lcm}
\opn\set{set}
\opn\ara{ara}
\opn\bight{bight}

\opn\div{div}
\opn\Div{Div}
\opn\cl{cl}
\opn\Cl{Cl}
\opn\Spec{Spec}
\opn\Supp{Supp}
\opn\supp{supp}
\opn\Sing{Sing}
\opn\Ass{Ass}
\opn\Min{Min}

\opn\Ann{Ann}
\opn\Rad{Rad}
\opn\Soc{Soc}
\opn\Ker{Ker}
\opn\Coker{Coker}
\opn\Im{Im}
\opn\Hom{Hom}
\opn\Tor{Tor}
\opn\Ext{Ext}
\opn\End{End}
\opn\Aut{Aut}
\opn\id{id}

\opn\nat{nat}
\opn\GL{GL}
\opn\SL{SL}
\opn\mod{mod}
\opn\ord{ord}
\opn\depth{depth}
\opn\set{set}
\opn\Shad{Shad}
\opn\pd{pd}
\opn\aff{aff}
\opn\con{conv}
\opn\relint{relint}
\opn\st{st}
\opn\lk{lk}
\opn\cn{cn}
\opn\core{core}
\opn\vol{vol}
\opn\gr{gr}

\opn\d{d}
\opn\Ind{Ind}
\opn{\indmat}{indmat}
\opn\diam{diam}
\opn\set{\mbox{\rm{set}}}
\opn\cochord{\mbox{\rm{cochord}}}
\opn \kk{{\Bbbk}}
\opn \mb{{\mathbf{m}}}
\opn\xb{{\mathbf{x}}}
\opn\dist{\mbox{\rm{d}}}

\def\pot#1#2{#1[\kern-0.28ex[#2]\kern-0.28ex]}

\opn\dirlim{\underrightarrow{\lim}}
\opn\invlim{\underleftarrow{\lim}}
\def\pnt{{\raise0.5mm\hbox{\large\bf.}}}

\def\Implies{\ifmmode\Longrightarrow \else
     \unskip${}\Longrightarrow{}$\ignorespaces\fi}
\def\implies{\ifmmode\Rightarrow \else
     \unskip${}\Rightarrow{}$\ignorespaces\fi}
\def\iff{\ifmmode\Longleftrightarrow \else
     \unskip${}\Longleftrightarrow{}$\ignorespaces\fi}

\let\:=\colon
\newtheorem{Theorem}{Theorem}[section]
\newtheorem{Lemma}[Theorem]{Lemma}
\newtheorem{Corollary}[Theorem]{Corollary}
\newtheorem{Proposition}[Theorem]{Proposition}
\newtheorem{Remark}[Theorem]{Remark}

\newtheorem{Example}[Theorem]{Example}

\newtheorem{Definition}[Theorem]{Definition}

\let\epsilon=\varepsilon
\let\phi=\varphi
\let\kappa=\varkappa
\title{The Line graph of a tree and its edge ideal}
\author{Anda Olteanu}
\address{Faculty of Marine Engineering, ``Mirceal cel B\u atr\^an" Naval Academy, Fulgerului Street, no. 1
	900218 Constanta, Romania,} \email{olteanuandageorgiana@gmail.com}
\begin{document}
	
	\maketitle
	\begin{abstract}  We describe all the trees with the property that the corresponding edge ideal of their line graph has a linear resolution. As a consequence, we give a complete characterization of those trees $T$ for which the line graph $L(T)$ is co-chordal. We compute also the second Betti number of the edge ideal of $L(T)$ and we determine the number of cycles in $\overline{L(T)}$. As a consequence, we obtain also the first Zagreb index of a graph. For edge ideals of line graphs of caterpillar graphs we determine the Krull dimension, the Castelnuovo-Mumford regularity, and the projective dimension under some additional assumption on the degrees of the cutpoints.
	\end{abstract}

\section*{Introduction}
Firstly used under this name by Harary and Norman in \cite{HN}, line graphs have been intensively studied in combinatorics. Recall that, for a finite simple graph $G$, its line graph, denoted by $L(G)$, is the graph with the vertex set given by the edges of $G$ and two vertices in $L(G)$ are joined by an edge if the corresponding edges are adjacent in $G$. 

There are several characterizations of graphs which are lines of some graphs. For instance, Beineken characterized the line graphs in terms of the forbidden induced subgraphs \cite{B}. Moreover, combinatorial properties of line graphs have been studied: Akiyama solved seven graph equations which involved line graphs, powers of graphs and the complementary of a graph \cite{A} and Milani\v{c}, Oversberg and Schaudt gave a characterization of those line graphs which are squares of graphs \cite{MOS}. Moreover, properties of line graphs have been determined in \cite{C,Ch,LC,VKA}. 

Recently, from the commutative algebra and algebraic topology point of view algebraic properties of the clique complex of the line graphs \cite{N} and their topology \cite{GSS} have been studied. From the combinatorial point of view, it is of interest to determine which properties of the graph $G$ are preserved by the line graph $L(G)$. For instance, it is known that the property of $G$ of being chordal is not preserved by its line graph \cite{C}. 

In this paper we consider edge ideals of the line graphs of trees and determine algebraic and homological properties which are expressed in terms of the original tree. 

The structure of the paper is the following: in the first section we recall all the necessary notions and results both from graph and commutative algebra. The second section is devoted to the study of the algebraic and homological invariants of edge ideals of line graphs of trees. We give a complete characterization of the trees $T$ for which the edge ideal of their line graph has a linear resolution [Theorem~\ref{linres}]. Moreover, for a tree $T$, we compute the second Betti number of edge ideals of line graphs of $T$ which will be very useful in computing the number of induced cycles in $\overline{L(T)}$.
In the last section we pay attention to caterpillar graphs, which are a particular class of trees which are of high interest in combinatorics. For edge ideals of line graphs of caterpillar trees, we compute the Castelnuovo--Mumford regularity, the projective dimension and the Krull dimension under some additional assumptions on the degrees of the cutpoints. These results allow us to determine the sizes of the largest and the minimal vertex cover and of the largest induced matching. In the end of the paper, we consider several remarks that arise naturally on the directions that one could consider. . 
    
\section{Preliminaries}
We review some standard facts on graph theory and edge ideals and we setup the notation and terminology that will be used through the paper. For more details, one may see \cite{DHS,F,HaTu1,HeHi,MV,Vi}. 
\subsection{Notions from graph theory}
Let $G$ be a finite simple graph with the vertex set $V(G)$ and the set of edges $E(G)$. Two vertices $u,v\in V(G)$ are called \textit{adjacent} (or \textit{neighbors}) if they form an edge in $G$. For a vertex $u$ of $G$, we denote by $\mathcal{N}(u)$ the set of all the neighbors of $u$, also called the \textit{neighborhood} of $u$. More precisely, $\mathcal{N}(u)=\{v\in V(G)\,:\,\{u,v\}\in E(G)\}$. \textit{The degree of the vertex $u$}, denoted by $\deg u$, is defined to be the size of the neighborhood set of $u$, that is $\deg u=|\mathcal{N}(u)|$. By \textit{a free vertex} we mean a vertex of degree $1$. A \textit{pendant edge} (or a \textit{whisker}) is an edge which contains a free vertex. A graph is called \textit{complete} if it has the property that any two vertices are adjacent. We denote by $\mathcal{K}_n$ the complete graph with $n$ vertices. 

By \textit{a subgraph} $H$ of $G$ we mean a graph with the property that $V(H)\subseteq V(G)$ and $E(H)\subseteq E(G)$. One says that a subgraph $H$ of $G$ is \textit{induced} if whenever $u,v\in V(H)$ so that $\{u,v\}\in E(G)$ then $\{u,v\}\in E(H)$. \textit{A clique} in $G$ is an induced subgraph which is a complete graph. \textit{A bridge} of a connected graph $G$ is an edge whose removal disconnects $G$, while a \textit{cutpoint} of $G$ is a vertex $u$ of $G$ such that the removal of $u$ and all its incident edges results in a disconnected graph.

\textit{A path of length }$t\geq2$ in $G$ is, by definition, a set of distinct vertices $u_0,u_1,\ldots,u_t$ such that $\{u_i,u_{i+1}\}$ are edges in $G$ for all $i\in\{0,\ldots,t-1\}$. \textit{The distance between two vertices $u$ and $v$ in $G$}, denoted by $\dist_G(u,v)$, is defined to be the length of a shortest path joining $u$ and $v$. If there is no path joining $u$ and $v$, then $\dist_G(u,v)=\infty$. We will skip the name of the graph when no confusion can occur. \textit{The diameter of the graph $G$}, denoted by $\diam(G)$, is defined to be the maximum of all the distances between any two vertices in $G$, namely $$\diam(G)=\max\{\dist(u,v): \ u,v\in V(G)\}.$$

\textit{A cycle of length $n\geq3$}, usually denoted by $C_n$, is a graph with the vertex set $[n]=\{1,\ldots,n\}$ and the set of edges $\{i,i+1\}$, where $n+1=1$ by convention. A graph is \textit{chordal} if it does not have any induced cycles of length strictly greater than $3$. A graph is called a \textit{tree} if it is connected and it does not have cycles. It is easy to see that any tree is a chordal graph. Moreover the vertices of a tree are either free or cutpoints.

For a graph $G$, we denote by $\overline{G}$ \textit{the complement of the graph} $G$, that is the graph with the same vertex set as $G$ and $\{u,v\}$ is an edge of $\overline{G}$ if it is not an edge of $G$. A graph $G$ is called \textit{co-chordal} if $\overline{G}$ is a chordal graph. One says that the edges $\{x,y\}$ and $\{u,v\}$ of $G$ form \textit{an induced gap} in $G$ if $x,y,u,v$ are the vertices of an induced cycle of length $4$ in $\overline{G}$. A graph $G$ is called \textit{gap-free} if it does not contain any induced gap.

	Let $G=(V(G),E(G))$ be a finite simple graph. \textit{The line graph} of the graph $G$, denoted by $L(G)$, is defined to have as its vertices the edges of $G$, and two vertices in $L(G)$ are adjacent if the corresponding edges in $G$ share a vertex in $G$.

There are several characterizations of those graphs which are line graphs of a graph. We recall here the one that will be used through the paper, but one can see \cite{B} for a characterization in terms of the forbidden induced subgraphs:
\begin{Proposition}\cite{K}\label{kn}
	A graph $G$ is a line graph if the edges of $G$ can be partitionated into maximal complete subgraphs such that no vertex lies in more than two of the subgraphs
\end{Proposition}

Since we will use the degrees of the vertices (both in $G$ and $L(G)$), we will recall here some formal definitions: 
\begin{Definition}\cite{Ch}\label{degedge}\rm
	Let $G$ be a graph and $e=\{u,v\}$ an edge. \textit{The degree of $e$} in $G$ is $\deg_G e=\deg u+\deg v-2$. Looking at $e$ as a vertex in $L(G)$, \textit{the degree of the vertex $e$ in $L(G)$} is equal with the degree of the edge $e$ in $G$.
\end{Definition}
\begin{Proposition}\cite[Proposition 1]{Ch}
	A necessary and sufficient condition that a vertex $w$ of the line graph $L(G)$ of a connected graph $G$ be a cutpoint is that it corresponds to a bridge $e = \{u,v\}$ of $G$ in which neither of the vertices $u$ and $v$ has degree one.
\end{Proposition}

\begin{Proposition}\cite[Proposition 2]{Ch}
	A necessary and sufficient condition that an edge $e = \{e_1,e_2\}$ be a bridge of the line graph $L(G)$ of a connected graph $G$ is that the edges $e_1$ and $e_2$ in $G$ be bridges in $G$ which meet in a vertex of degree two.
\end{Proposition}

\subsection{Edge ideals}

Given a finite simple graph $G$ with the vertex set $V(G)=\{1,\ldots,n\}=[n]$ and the set of edges $E(G)$, one may consider its \textit{edge ideal} which is the squarefree monomial ideal denoted by $I(G)\subseteq S=\kk[x_1,\ldots,x_n]$, where $\kk$ is a field, defined by $I(G)=\langle x_ix_j\ :\ \{i,j\}\in E(G)\rangle $. 

Edge ideals have been intensively studied and properties of invariants such that Betti numbers, projective dimension, Castelnuovo--Mumford regularity, or Krull dimension have been established for several classes of graphs (see \cite{HeHi,MV,Vi} for more details). We recall that, if $I\subseteq S=\kk[x_1,\ldots,x_n]$ is an ideal and $\mathcal{F}$ is the minimal graded free resolution of $S/I$ as an $S$-module:
\[\mathcal{F}: 0\rightarrow\bigoplus\limits_jS(-j)^{\beta_{pj}}\rightarrow\cdots\rightarrow\bigoplus\limits_j S(-j)^{\beta_{1j}}\rightarrow S\rightarrow S/I\rightarrow0,\] then the numbers $\beta_{ij}$ are \textit{the graded Betti numbers of $S/I$}, \textit{the projective dimension of $S/I$} is \[\projdim \,S/I=\max\{i:\beta_{ij}\neq 0\}\] and \textit{the Castelnuovo--Mumford regularity} is \[\reg\, S/I=\max\{j-i:\beta_{ij}\neq0\}.\]
Let $d>0$ be an integer. An ideal $I$ of $S$ \textit{has a $d$--linear resolution} if the minimal graded free resolution of $I$ is of the form
\[\ldots\longrightarrow S(-d-2)^{\beta_2}\longrightarrow S(-d-1)^{\beta_2}\longrightarrow S(-d)^{\beta_1}\longrightarrow S\longrightarrow S/I\longrightarrow 0.
\] Equivalently, an ideal $I$ has a $d$-linear resolution if and only if it is minimally generated in degree $d$ and  $\reg\, S/I=d-1$. If $d=2$, we simply say that the ideal has a linear resolution.

Fr\"oberg's Theorem gives a combinatorial characterization of the property of an edge ideal to have a linear resolution:
\begin{Theorem}[Fr\"oberg]\cite{F} \label{Froberg} Let $G$ be a finite simple graph. The edge ideal $I(G)$ has a linear resolution if and only if $G$ is a co-chordal graph.
\end{Theorem}
The following remark is a direct consequence of the theorem of Fr\"oberg and it will be intensively used through the paper:
\begin{Remark}\label{gapfree}\rm
	If $G$ is a graph such that its edge ideal has a linear resolution, then $G$ is gap-free.
\end{Remark}

In the sequel we recall connections between Krull dimension, regularity and projective dimension of the edge ideal and combinatorial invariants of the graph. We start with the Krull dimension. 

A subset $W$ of $V(G)$ is called \textit{an independent set of $G$} if it does not contain any edge of $G$, i.e. for all $u,v\in W$ such that $u\neq v$, one has $\{u,v\}\notin E(G)$. 
One may compute the Krull dimension of $S/I(G)$ by using independent sets (see \cite[Section 2]{MV} or \cite[Lemma 1]{HM} for more details):
\begin{Proposition}\cite{HM}\label{dimindep}
	$\dim S/I(G)=\max\{|W|:\ W\mbox{ is an independent set of }G\}.$
\end{Proposition}
A subset $M=\{e_1,\ldots,e_s\}$ of $E(G)$ is called \textit{a matching of $G$} if any two edges from $M$ are disjoint, i.e. if for all $i\neq j$, one has $e_i\cap e_j=\emptyset$. \textit{An induced matching} in $G$ is a matching which forms an induced subgraph of $G$.\textit{ The induced matching number} of $G$, denoted by $\indmat(G)$, is defined to be the number of edges in a largest induced matching, that is 
$$\indmat(G)=\max\{|M|:\ M\mbox{ is an induced matching in }G\}.$$

	\textit{A clique-neighborhood} is the set $K$ of edges of a clique together with some edges each of which is incident to a member of $K$.
For chordal graphs, in between the number of edges of the largest induced matching and the smallest number of sets of clique-neighborhoods there is the following connection:
\begin{Theorem}\cite[Theorem 2]{C}\label{ind-CN}
	For a chordal graph $G$, \[\max\{|M|: M\mbox{ is an induced matching in }G\}=\]\[=\min\{|N|: N\mbox{ is a set of clique-neighborhoods in }G\mbox{ which covers }E(G)\}.\]
\end{Theorem}

Another combinatorial invariant of the graph $G$ that will be used is\textit{the co-chordal cover number of $G$}, denoted by $\cochord(G)$, which is the minimum number of co-chordal subgraphs required to cover the edges of $G$ that is \[\cochord(G)=\min\{s\in{\NN}: \mbox{ there are }G_1,\ldots,G_s\mbox{ co-chordal subgraphs of }G\]\[\mbox{ such that } E(G)=E(G_1)\cup\cdots\cup E(G_s)\}\]

In between the induced matching number, the co-chordal cover number, and the Castelnuovo--Mumford regularity there are the following connections:
\begin{Proposition}\cite[Theorem 1]{W}\label{indmatgen}
	Let $G$ be a finite simple graph. Then, over any field $\kk$, $\reg\, S/I(G)\leq\cochord(G)$.
\end{Proposition}
A lower bound for the Castelnuovo--Mumford regularity of $I(G)$ is given by the induced matching number.
\begin{Proposition}\cite[Lemma 2.2]{Ka}\label{regindmat}
	For any graph $G$, we have $\reg\, S/I(G)\geq\indmat(G)$.
\end{Proposition}
\begin{Proposition}\cite[Corollary 6.9]{HaTu}\label{reg} If $G$ is a chordal graph, then $\reg\, S/I(G)=\indmat(G)$.
\end{Proposition}
A different upper bound for the Castelnuovo--Mumford regularity of the edge ideal of a graph $G$ can be given in terms of maximal induced cliques of $G$:
\begin{Proposition}\cite[Theorem 2]{W}\label{regcliques}
	If $G$ is a graph such that $V(G)$ can be partitioned into
	an (induced) independent set $J_0$ together with $s$ cliques $J_1,\ldots, J_s$, then
	$\reg \,S/I(G) \leq s$.
\end{Proposition}
The following result describes the behavior of the Castelnuovo--Mumford re\-gularity with respect to induced subgraphs.
\begin{Proposition}\cite[Proposition 3.8]{MV}\label{regind}
	If $H$ is an induced subgraph of $G$, then $\reg\, S/I(H)\leq\reg\, S/I(G)$.
\end{Proposition}

\section{The line graph of a tree}
Throughout this section we consider general properties of algebraic and homological invariants of edge ideals of the line graph of trees. For a tree $T$, we study the behavior of the Castelnuovo--Mumford regularity when we delete a vertex from $T$, and we will pay attention to the property of the edge ideal $I(L(T))$ of having a linear resolution. Since we are dealing with both the tree $T$ and its line graph $L(T)$, throughout this paper we will assume that $I(T)\subseteq S=\kk[x_1,\ldots,x_n]$ and $I(L(T))\subseteq R=\kk[e_{uv}:\ \{u,v\}\in E(G)]$.

For trees, there is the following characterization of their line graph. 
\begin{Lemma}\cite{Ch} A graph is the line graph of a tree if and only if it is a connected block graph in which each cutpoint is on exactly two blocks.
\end{Lemma}
Recall that \textit{a block graph} is a connected graph in which every block (maximal biconnected induced subgraph) is a clique.
\begin{Remark}\rm
If $T$ is a tree, then $L(T)$ is a chordal graph.
\end{Remark} 
Note that the above remark and Fr\"oberg's Theorem \ref{Froberg}
allow us to determine the induced matching of $\overline{L(T)}$, where $T$ is a tree:\begin{Proposition}
	Let $T$ be a tree. Then $\indmat(\overline{L(T)})=1$.
\end{Proposition}
\begin{proof} Since $T$ is a tree, its line graph is chordal, therefore $\overline{L(T)}$ is co-chordal. Thus, by Fr\"oberg's Theorem \ref{Froberg}, $I(\overline{L(T)})$ has a linear resolution. The statement follows by Proposition \ref{regindmat}.
	\end{proof}
We will consider next the behavior of the regularity of the line graph when we delete a pendant edge.
\begin{Proposition}
	Let $T$ be a tree and $u$ a free vertex of $T$ such that $\{u,v\}$ is an edge, $\deg_T v\geq4$ and each vertex from $\mathcal{N}_T(v)$ has degree either one or at least three. Let $T'=T\setminus\{u\}$.  
	\begin{itemize}
		\item[a)] If $v$ has at least $3$ free vertices then $\indmat L(T)=\indmat L(T')$.
		\item[b)]  If $v$ has $2$ free vertices, then $\indmat L(T)=\indmat L(T')+1$.
		\item[c)] If $v$ has only one free vertex, then $\indmat L(T)=\indmat L(T')$.
		\end{itemize}
\end{Proposition}
\begin{proof} Assume that $\indmat L(T)=s$. By Theorem \ref{ind-CN}, it follows that all the edges of $L(T)$ can be covered by at least $s$ clique-neighborhoods. Let $\deg_T(v)=d\geq 4$.
	
	a) Since $v$ has at least three free vertices, the clique-neighborhood induced by $v$ and its neighbors in $L(T)$ must be in the considered minimal set of clique-neighborhoods since all the edges of $L(T)$ should be covered. The vertex $v$ and its neighbors will give in $L(T)$ a maximal clique of size $d\geq4$ which is connected to at most $d-3$ cliques. In $L(T')$, the same vertex and its neighbors will give a clique of degree $d-1$ which is connected to at most $d-3$ cliques. Hence, the number of connected cliques is not sufficient to cover all the edges of the clique which is given by $v$ and its neighbors, so the number of required clique-neighborhoods does not decreases. Therefore $\indmat L(T')=s$.
	
	b) Since $v$ has at least two free vertices, the clique-neighborhood induced by $v$ and its neighbors in $L(T)$ must be in the considered minimal set of clique-neighborhoods since all the edges of $L(T)$ should be covered. The vertex $v$ and its neighbors will give in $L(T)$ a maximal clique of size $d\geq4$ which is connected to  $d-2$ cliques. In $L(T')$, the same vertex will give a clique of degree $d-1$ which is connected to $d-2$ cliques. Hence, the number of connected cliques is sufficient to cover all the edges of the clique which is given by $v$ and its neighbors, so the number of required clique-neighborhoods decreases by one. Therefore $\indmat L(T')=s-1$.
	
	c) The vertex $v$ will give in $L(T)$ a maximal clique of size $d\geq4$ which is connected to $d-1$ cliques which is not in the set of clique-neighborhoods (due to our assumption on the degrees of the vertices from $\mathcal{N}(v)$). In $L(T')$, the same vertex will give a clique of degree $d-1$ which is connected to $d-1$ cliques. Therefore, the number of required clique-neighborhoods is not changed. Hence $\indmat L(T')=s$.
	\end{proof}
Since $L(T)$ is a chordal graph, the next corollary follows by Proposition \ref{reg}:
\begin{Corollary}
	Let $T$ be a tree and $u$ a free vertex of $T$ such that $\{u,v\}$ is an edge, $\deg_T v\geq4$ and each vertex from $\mathcal{N}_T(v)$ has degree either one or at least three. Let $T'=T\setminus\{u\}$.  
	\begin{itemize}
		\item[a)] If $v$ has at least $3$ free vertices then $\reg\, R/I(L(T))=\reg\, R/I(L(T'))$.
		\item[b)]  If $v$ has $2$ free vertices, then $\reg\, R/I(L(T))=\reg\, R/I(L(T'))+1$.
		\item[c)] If $v$ has only one free vertex, then $\reg\, R/I(L(T))=\reg\, R/I(L(T'))$.
	\end{itemize}
\end{Corollary}

\begin{Proposition}
	Let $T$ be a tree which is not a star and $u$ a free vertex of $T$ such that $\{u,v\}$ is an edge and $\deg_T v=3$. Let $T'=T\setminus\{u\}$. Then $$\indmat L(T)=\indmat L(T')+1.$$
\end{Proposition}
\begin{proof}

 We will use Theorem \ref{ind-CN} in order to prove the equality.
	The vertex $v$ and its neighbors yield in $L(T)$ a maximal clique of size $3$. In $L(T')$, the same vertex will give an edge which is connected to a clique. So the number of required clique-neighborhoods decreases by one.\end{proof}
The following corollary is straightforward:
\begin{Corollary}
	Let $T$ be a tree which is not a star and $u$ a free vertex of $T$ such that $\{u,v\}$ is an edge and $\deg_T v=3$. Let $T'=T\setminus\{u\}$. Then $$\reg\, R/I(L(T)))=\reg\, R/I(L(T')))+1.$$
\end{Corollary}

We will characterize now all the trees for which the edge ideal of their line graph has a linear resolution. Note that the property of having a linear resolution is not preserved by considering the line graph. 

\begin{Example}\label{example}\rm Let $T_1$ and $L(T_1)$ be the following tree and its line graph:

\begin{center}
	\begin{figure}[h]
		\includegraphics[height=2cm]{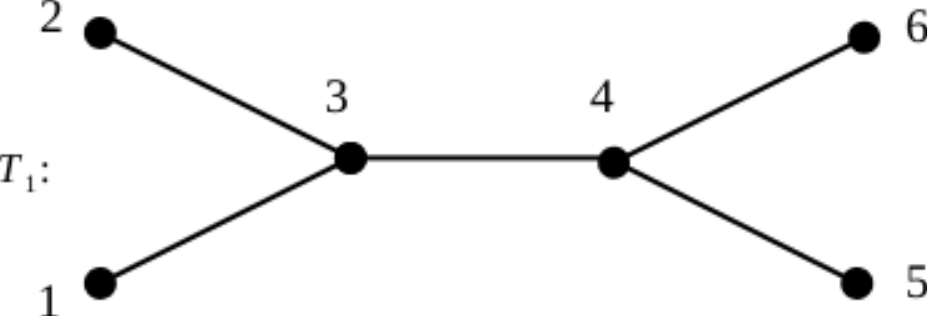}$\qquad$
		\includegraphics[height=2cm]{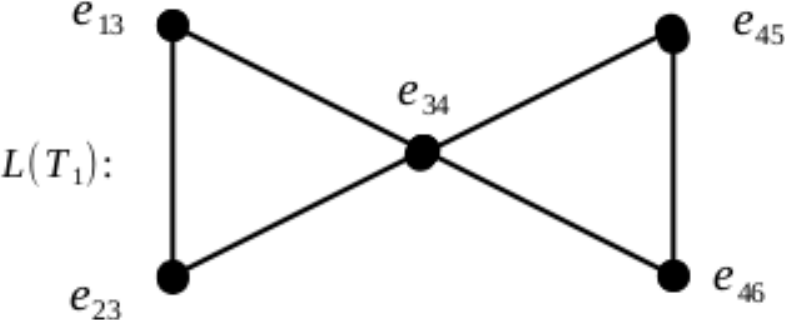}
	\end{figure}
\end{center}
Then $I(T_1)$ has a linear resolution (it is chordal with $\indmat T_1=1$) while $I(L(T_1))$ has not (it has an induced matching of size $2$, $\{e_{13},e_{23}\},\{e_{45},e_{46}\}$).
\end{Example}
We recall that \textit{a star} is the graph with the set of vertices $\{u,v_1,\ldots,v_n\}$ and with the edges $\{u,v_1\},\ldots,\{u,v_n\}.$ \textit{The path $P_n$} is the graph with $n$ vertices $u_1,\ldots,u_n$ and the edges $\{u_i,u_{i+1}\}$, for $1\leq i\leq n-1$. \textit{A broom} is the graph obtained from $P_n$ by appending $m$ new vertices and the corresponding pendant edges to the first (or last) vertex of $P_n$. \textit{A whiskered graph} is the graph obtained from the graph $G$ by adding to each of its vertices a pendant edge (together with a new vertex). The whiskered graph of $G$ is also known in graph theory as \textit{the corona of $G$}. If one adds a pendant edge (together with a new vertex) to a subset of $V(G)$, then the obtained graph is called \textit{a partially whiskered graph.} The graph $G$ is \textit{weakly chordal} if neither $G$, nor $\overline{G}$ do not have any induced cycle of length strictly greater than $4$. Note that any chordal graph is weakly chordal (see for instance \cite{CST}).
\begin{Theorem}\label{linres}
Let $T$ be a tree. Then $I(L(T))$ has a linear resolution if and only if $T$ satisfies one of the following conditions:

\begin{itemize}
	\item[i)]  a star graph;
	\item[ii)] a broom of diameter $3$
	\item[iii)] a (partially) whiskered star
	\item[iv)] $P_n$, $2\leq n\leq5$ 
\end{itemize}
\end{Theorem}
\begin{proof}
``$\Leftarrow$" Firstly we show that $I(L(T))$ has a linear resolution if $T$ is a star graph or is a broom of diameter $3$. If $T$ is a star, then $L(T)$ is a clique, therefore $I(L(T))$ has a linear resolution. If $T$ is a broom of diameter $3$ then $L(T)$ is a clique with one whisker, therefore $\overline{L(T)}$ is a star graph, hence it is chordal. According to Fr\"oberg's Theorem, $I(L(T))$ has a linear resolution.

 Next we consider the case when $T$ is a partially whiskered star. In this case $L(T)$ has a clique and some pendant edges (at most one to each vertex). Since $\diam (L(T))=3$, the line graph $L(T)$ does not have any induced gap. Therefore $\overline{L(T)}$ does not contain $C_4$ as an induced cycle. Moreover, $L(T)$ is chordal, thus it is weakly chordal and $\overline{L(T)}$ does not contain any induced cycle $C_k$, with $k\geq5$. Therefore $\overline{L(T)}$ is chordal and, by Fr\"oberg's Theorem \ref{Froberg}, $I(L(T))$ has a linear resolution. 
 
 Finally, for $P_n$, with $n\in\{2,3,4,5\}$, one may see that they are particular classes of whiskered stars, so $I(L(P_n))=I(P_{n-1})$ has a linear resolution. 

``$\Rightarrow$" Conversely, we assume that $I(L(T))$ has a linear resolution. If $\diam(T)\geq5$, then $T$ contains $P_6$ as an induced subgraph. Since $L(P_6)=P_5$, one has that $\reg(I(L(P_5)))=3$, therefore $\reg(I(T))\geq\reg(I(L(P_5)))=3$ by Proposition \ref{regind} and $I(T)$ does not have a linear resolution. Hence, $\diam (T)\leq 4$. It is clear that if $T$ has only one edge ($T=P_2$), then $I(L(T))$ has a linear resolution. We discuss now the remaining cases:

\textit{Case 1:} If $\diam T=2$, then $T$ is a star graph. 

\textit{Case 2:} Assume now that $\diam(T)=3$. According to Example \ref{example}, $T$ cannot contain an induced subgraph of the following form  
\begin{center}
\begin{figure}[h]
\includegraphics[height=2cm]{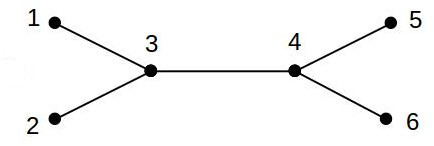}
\end{figure}
\end{center}since the edge ideal of its line graph does not have a linear resolution.
Therefore $T$ can be a broom graph or the path $P_4$.

\textit{Case 3:} If $\diam(T)=4$ then $T$ is either $P_5$ or $T$ can have at most one vertex of degree greater than or equal to $3$ (otherwise it contains an induced subgraph as in Example \ref{example}). Then $T$ can contain as an induced subgraph one of the following graphs: 
\begin{center}
\begin{figure}[h]
\includegraphics[height=2.6cm]{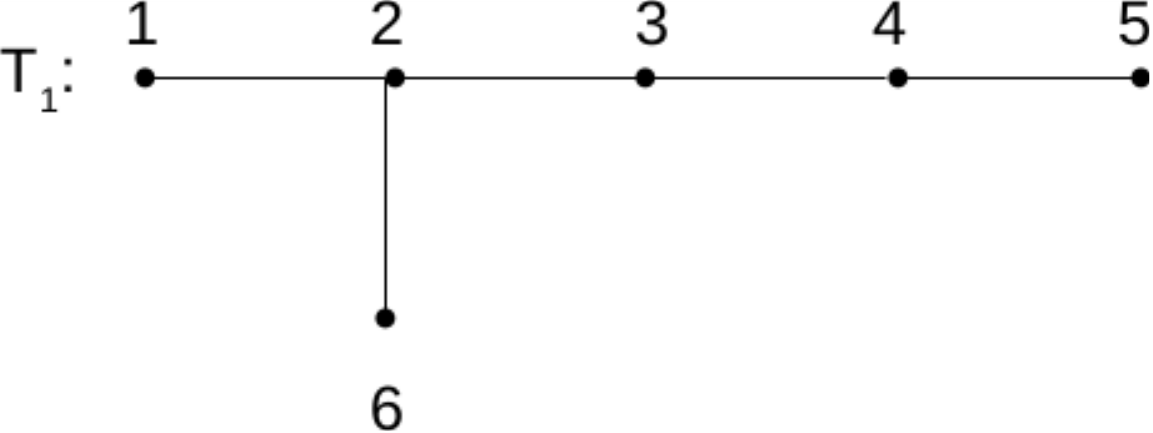}$\ \ \ $
\includegraphics[height=2.6cm]{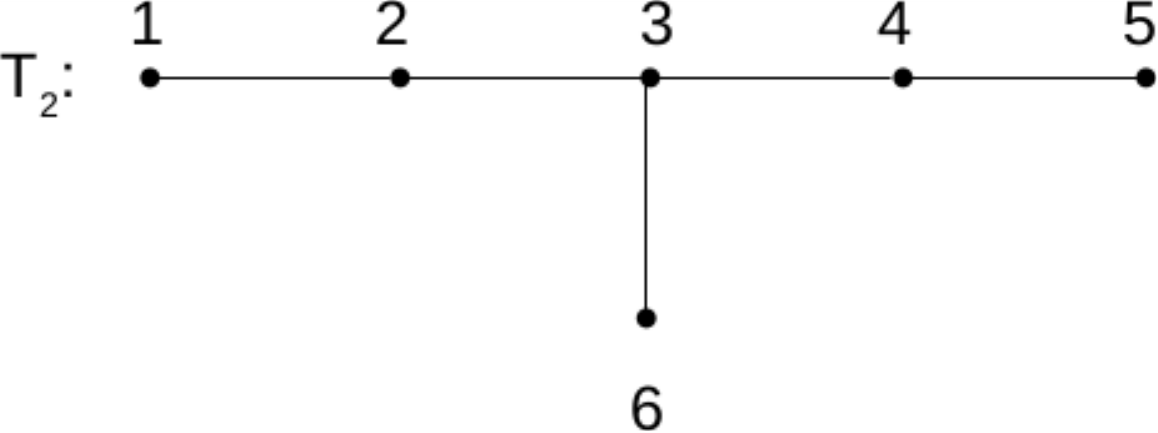}

\includegraphics[height=2.6cm]{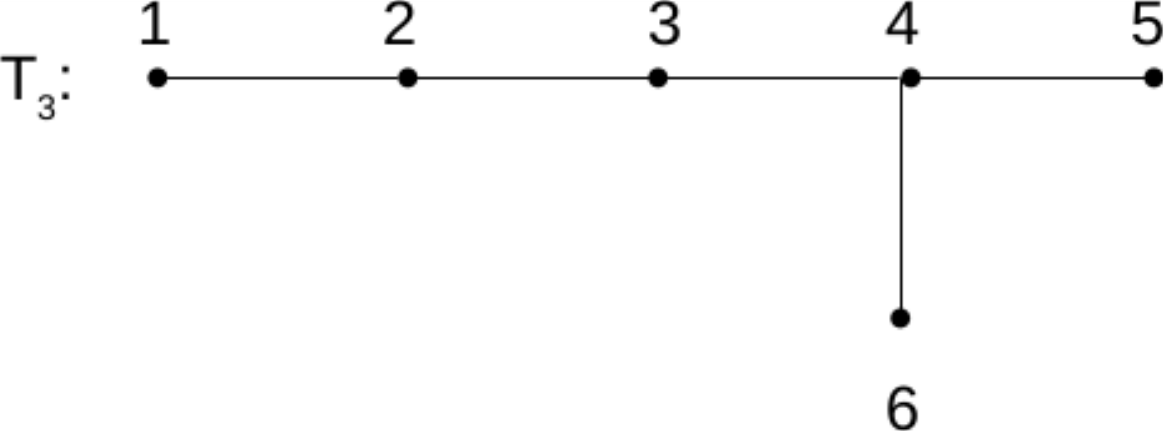}
\end{figure}
\end{center}
whose line graphs are

\begin{center}
\begin{figure}[h]
\includegraphics[height=2.6cm]{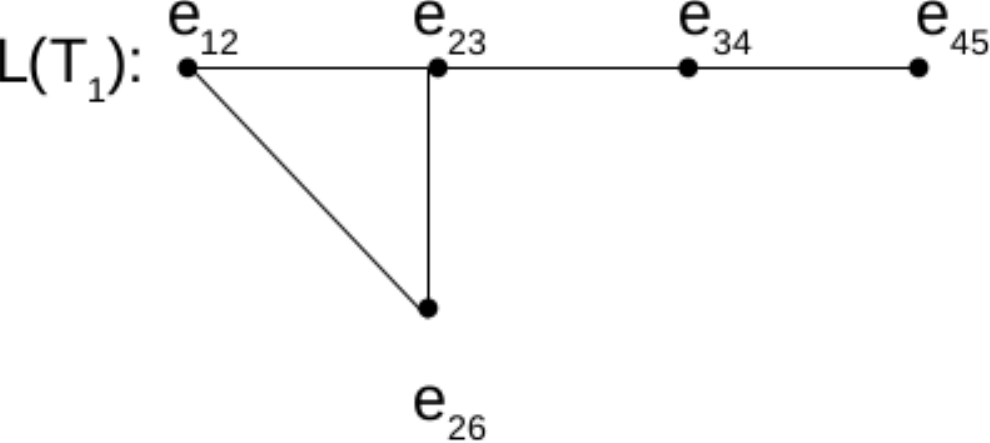}$\ \ \ $
\includegraphics[height=2.6cm]{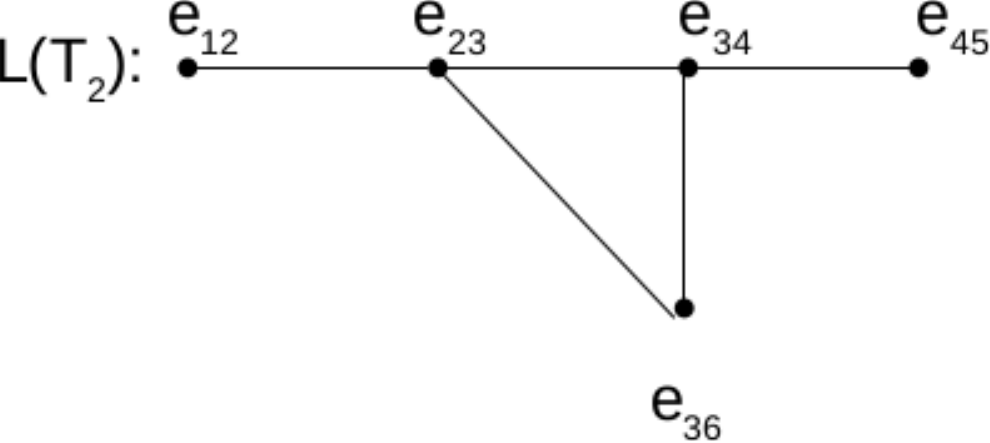}

\includegraphics[height=2.6cm]{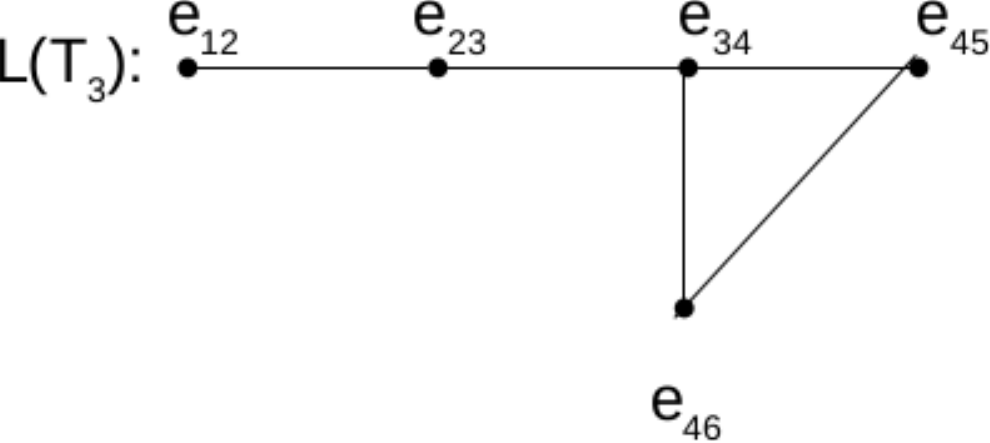}
\end{figure}
\end{center}
One may easily check that both $L(T_1)$ and $L(T_3)$ have an induced gap ($\{e_{12},e_{26}\}$ and $\{e_{34},e_{45}\}$ for $L(T_1)$, respectively $\{e_{12},e_{23}\}$ and $\{e_{46},e_{45}\}$ for $L(T_3)$), therefore, their edge ideals do not have a linear resolution. Hence $T$ is a (partially) whiskered star. 
\end{proof}
\begin{Remark}\rm
	Note that in the above theorem, broom graphs of diameter $3$ and path graphs $P_n$, with $2\leq n\leq 5$ are particular classes of partially whiskered stars. Still we consider them separately due to their importance. 
\end{Remark}
The following corollary follows by Fr\"oberg's Theorem:
\begin{Corollary}
Let $T$ be a tree. Then $L(T)$ is co-chordal if and only if $T$ is one of the following graphs
\begin{itemize}
	\item[i)]  a star graph;
	\item[ii)] a broom of diameter $3$
	\item[iii)] a (partially) whiskered star
	\item[iv)] $P_n$, $2\leq n\leq 5$ 
\end{itemize}
\end{Corollary}
By Proposition \ref{reg} and Fr\"oberg's Theorem \ref{Froberg}, one also have the following equivalence:
\begin{Corollary}
	Let $T$ be a tree. The following are equivalent:
	\begin{itemize}
		\item[a)] $\indmat(L(T))=1$;
		\item[b)] $\cochord(L(T))=1$;
		\item[c)] $T$ is one of the following graphs:
		\begin{itemize}
			\item[i)]  a star graph;
			\item[ii)] a broom of diameter $3$
			\item[iii)] a (partially) whiskered star
			\item[iv)] $P_n$, $n\in \{2,3,4,5\}$ 
		\end{itemize}
		
	\end{itemize}
\end{Corollary} 
In the sequel, we pay attention to determine the second graded Betti number of $I(L(T))$, where $T$ is a tree. Our results are expressed in terms of the combinatorial invariants of the tree $T$. The next two results will be extremely usefull. The first one computes the number of edges of $L(G)$, for an arbitrary finite simple graph $G$.
\begin{Proposition}\cite[Proposition 7.6.2]{Vi}\label{betti1}  If $G$ is a graph with vertices $u_1,\ldots,u_n$ and edge set $E(G)$, then the number of edges of the line graph $L(G)$ is given by
\[	|E(L(G))| = \sum\limits_{i=1}^n
{\deg u_i \choose2}	= -|E(G)| +\sum\limits_{i=1}^n\frac{\deg^2u_i}{2} .\]
	\end{Proposition}

In combinatorics, the first Zagreb index
$M_1 (G)$ of the graph $G$ is defined as $$M_1(G)=\sum\limits_{u\in V(T)}\deg^2 u,$$ see for instance \cite{DXN} for more details. By using Proposition~\ref{betti1}, one may express the first Zagreb index in terms of the number of edges of the graph $G$ and its line graph: 

\begin{Corollary}
	If $G$ is a graph with vertices $u_1,\ldots,u_n$ and edge set $E(G)$, then the first Zagreb number of the graph $G$ is given by
	\[ M_1(G)=2(|E(L(G))| +|E(G)|)=2\left(|E(G)|+\sum\limits_{i=1}^n
	{\deg u_i\choose2}\right).\]
\end{Corollary}
The next result determines the second Betti number of the edge ideal of a graph $G$:
\begin{Proposition}\cite{EV}\label{betti2}
	Let $I \subset S$ be the edge ideal of a graph $G$, let $V$	be the vertex set of $G$, and let $L(G)$ be the line graph of $G$. If
\[	\cdots\longrightarrow S^c(-4)\bigoplus S^b(-3) \longrightarrow S^q(-2) \longrightarrow S \longrightarrow S/I \longrightarrow 0\]
	is the minimal graded resolution of $S/I$. Then $$b = |E(L(G))|-N_t,$$ where
	$N_t$ is the number of triangles of $G$ and $c$ is the number of unordered pairs
	of edges $\{f,g\}$ such that $f\cap g = \emptyset$ and $f$ and $g$ cannot be joined by an edge.
\end{Proposition}

We will apply the above results for edge ideals of the line graph of a tree.
\begin{Corollary}
	If $T$ is a tree with vertices $u_1,\ldots,u_n$ and edge set $E(T)$, then the number of edges of the line graph $L(T)$ is given by
	\[	|E(L(T))| = \sum\limits_{i=1}^n
	{\deg_T u_i\choose2}	= 1-n +\sum\limits_{i=1}^n\frac{\deg_T^2u_i}{2}\]
	\end{Corollary}
	\begin{proof} The proof is straightforward since a tree on $n$ vertices has $n-1$ edges.
	\end{proof}
Next we determine the second Betti number for the edge ideal of the line graph of trees:
\begin{Proposition}
		Let $T$ be a tree with the vertex set $\{u_1,\ldots,u_n\}$, $L(T)$ its line graph, and $I(L(T))\subset R$. If
	\[	\cdots\longrightarrow R^c(-4)\bigoplus R^b(-3) \longrightarrow S^q(-2) \longrightarrow R \longrightarrow R/I(L(T)) \longrightarrow 0\]
	is the minimal graded resolution of $R/I(L(T))$, then $$b = \sum\limits_{\{u_i,u_j\}\in E(T)}{{\deg u_i+\deg u_j-2}\choose{2}}-\sum_{i=1}^n{{\deg u_i}\choose{3}}.$$
\end{Proposition}
\begin{proof} According to Proposition \ref{betti2}, $b=|E(L^2(T))|-N_t(L(T))$, where we denote $L^2(T)=L(L(T))$. Note that, by Proposition \ref{betti1}, \[|E(L^2(T))|=\sum_{e\in V(L(T))}{{\deg_{L(T)} e}\choose 2}=\sum_{\{u_i,u_j\}\in E(T)}{{\deg u_i+\deg u_j-2}\choose 2}\] according to the Definition \ref{degedge}. The number of triangles of $L(T)$ is $$ N_t(L(T))=\sum_{i=1}^n{{\deg u_i}\choose{3}}$$ since a triangle in $L(T)$ is provided by $3$ neighbors of a vertex of $T$ of degree at least $3$. The statement follows.
	\end{proof}

\begin{Proposition}\label{betti2-c}
	Let $T$ be a tree with $n$ vertices, $L(T)$ the line graph of $T$, and $I(L(T))\subset R$. If
	\[	\cdots\longrightarrow R^c(-4)\bigoplus R^b(-3) \longrightarrow S^q(-2) \longrightarrow R \longrightarrow R/I(L(T)) \longrightarrow 0\]
	is the minimal graded resolution of $R/I(L(T))$. Then $$c = \sum\limits_{\begin{matrix} u,v\in V(T),\\ \d(u,v)\neq2\end{matrix}}{{|\mathcal{N}(u)\setminus \{v\}|}\choose{2}}\cdot{{|\mathcal{N}(v)\setminus \{u\}|}\choose{2}}+$$ $$+\sum\limits_{\begin{matrix}
			u,v\in V(T),\\ \d(u,v)=2
			\end{matrix}}\left[{{|\mathcal{N}(u)|}\choose{2}}\cdot{{|\mathcal{N}(v)|}\choose{2}}-(|\mathcal{N}(u)|-1)(|\mathcal{N}(v)|-1)\right].$$
\end{Proposition}

\begin{proof}
	By Proposition \ref{betti2-c}, $c$ is the number of pairwise disjoint edges $f$ and $g$ in $L(T)$, which are not connected by any edge. Taking into account the relation between $T$ and $L(T)$, we can consider the following cases:
	
	\textit{Case 1:} There are two vertices $u$ and $v$ such that $\{u,v\}\in E(T)$ and $\{u,v\}$ is not a leaf in $T$. We may assume that the induced subgraph is of the form
	
	\begin{center}
		\begin{figure}[h]
			\includegraphics[height=2cm]{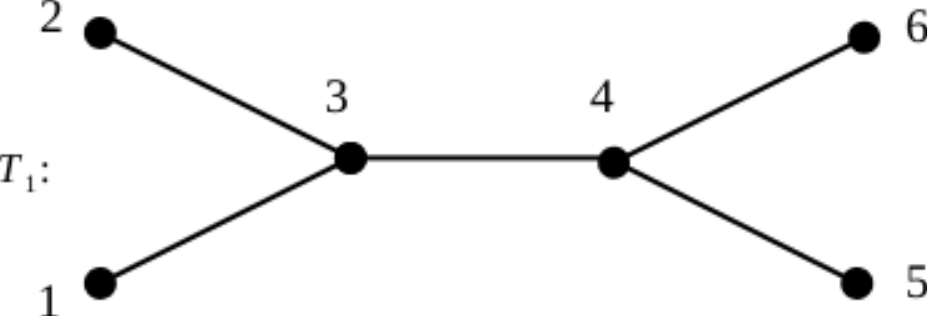} $\qquad$
			\includegraphics[height=2cm]{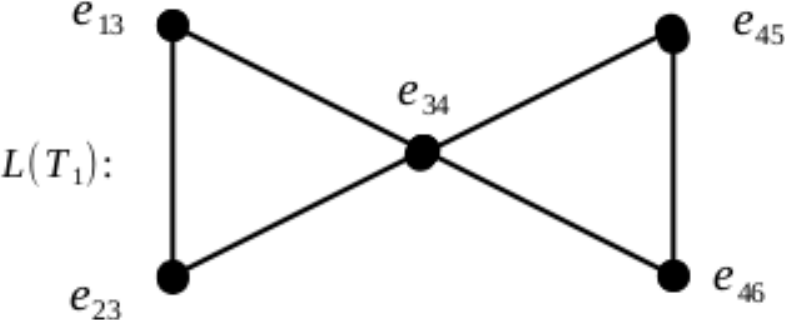}
		\end{figure}
	\end{center}
where $\{u,v\}=\{3,4\}$.
	Therefore we have only one gap in $L(T_1)$, namely $\{e_{13},e_{23}\},$ $\{e_{45},e_{46}\}$. Note that this gap is induced by the vertices $1,2,5,6$ from $T$. Therefore, in general if $\mathcal{N}(u)$ and $\mathcal{N}(v)$ are the set of neighbors of $u$ and $v$, a gap can be obtained by the edges $\{u_1,u_2\},\{v_1,v_2\}$ where $u_1,u_2\in \mathcal{N}(u)\setminus \{v\}$ and $v_1,v_2\in \mathcal{N}(v)\setminus \{u\}$. Therefore, vertices $u$ and $v$ give $\displaystyle{{|\mathcal{N}(u)\setminus \{v\}|}\choose{2}}\cdot{{|\mathcal{N}(v)\setminus \{u\}|}\choose{2}}$ gaps in $L(T)$. Note that, if $u$ or $v$ are of degree 2, their neighbors do not induce any gap in $L(T)$.
	
	\textit{Case 2:} The next case is the one when $u$ and $v$ are vertices of $T$ such that $\d_T(u,v)=2$ and they are not free vertices. We may assume that $T$ contains the following graph as an induced subgraph (where $u=3$ and $v=5$)
	\begin{center}
		\begin{figure}[h]
			\includegraphics[height=2.3cm]{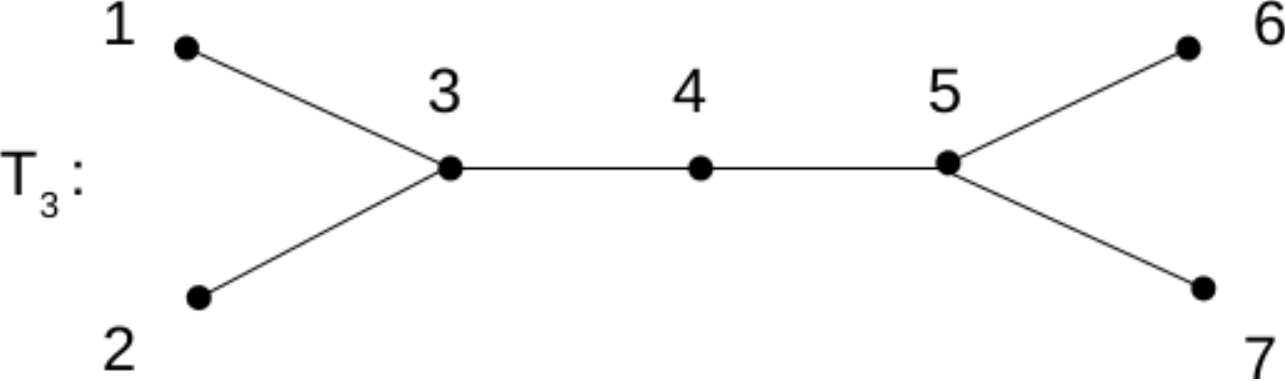} \end{figure}
		\end{center}
	whose line graph is
	\begin{center}
		\begin{figure}[h]
			\includegraphics[height=2.3cm]{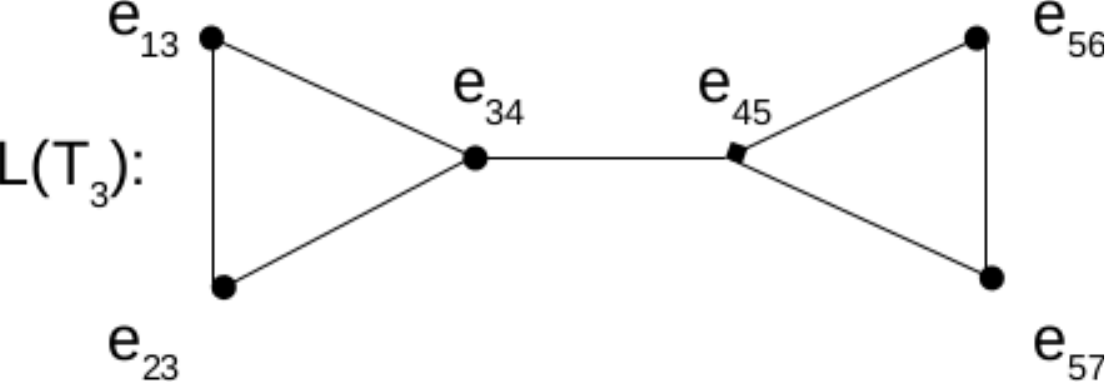}
		\end{figure}
	\end{center}
	 The total number of pairs that can be formed by using the edges which come from the neighbors of vertices $3$ and $5$ is $\displaystyle{{|\mathcal{N}(3)|}\choose{2}}\cdot{{|\mathcal{N}(5)|}\choose{2}}$. The gaps are obtained by taking the edge $\{e_{13},e_{23}\}$ and any edge from the right triangle or the edge $\{e_{56},e_{57}\}$ and any edge from the left triangle. Therefore, we have to remove $|\mathcal{N}(3)-1|\cdot|\mathcal{N}(5)-1|$ pairs of edges. Therefore, in general, these type of vertices $u$ and $v$ yield $${{|\mathcal{N}(u)|}\choose{2}}\cdot{{|\mathcal{N}(v)|}\choose{2}}-(|\mathcal{N}(u)|-1)(|\mathcal{N}(v)|-1)$$ gaps in $L(T)$. The statement follows.
	
	\textit{Case 3:} Let's assume now that $u$ and $v$ are vertices of $T$ such that $\d_T(u,v)\geq3$ and they are not free vertices. We will consider the case when distance is $3$, but the arguments are valid also for higher distances. 
	We may assume that $T$ contains the following graph as an induced subgraph (where $u=3$ and $v=6$)
	\begin{center}
		\begin{figure}[h]
			\includegraphics[height=2.5cm]{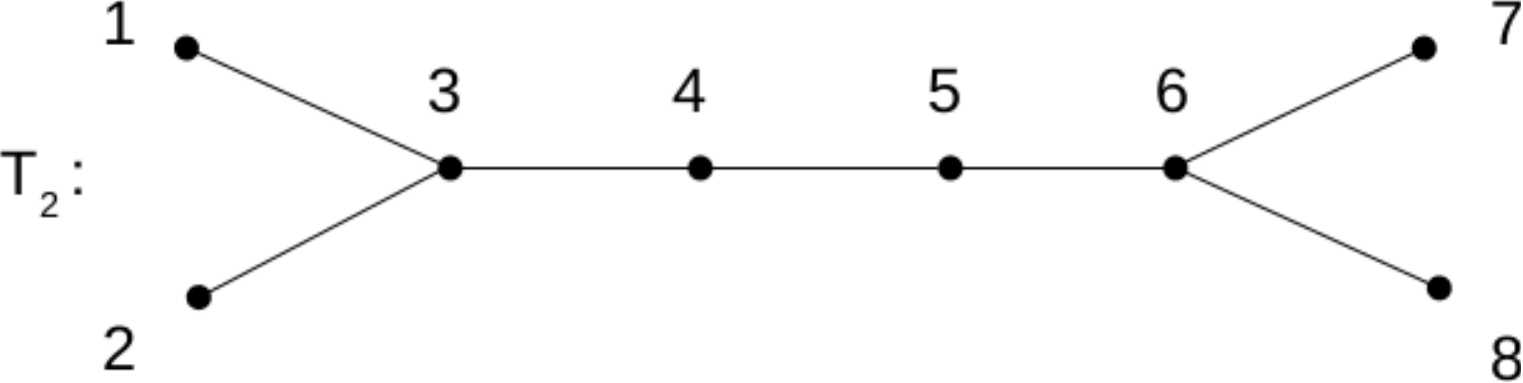} 
		\end{figure}
	\end{center}
	whose line graph is
	\begin{center}
		\begin{figure}[h]
					\includegraphics[height=2.5cm]{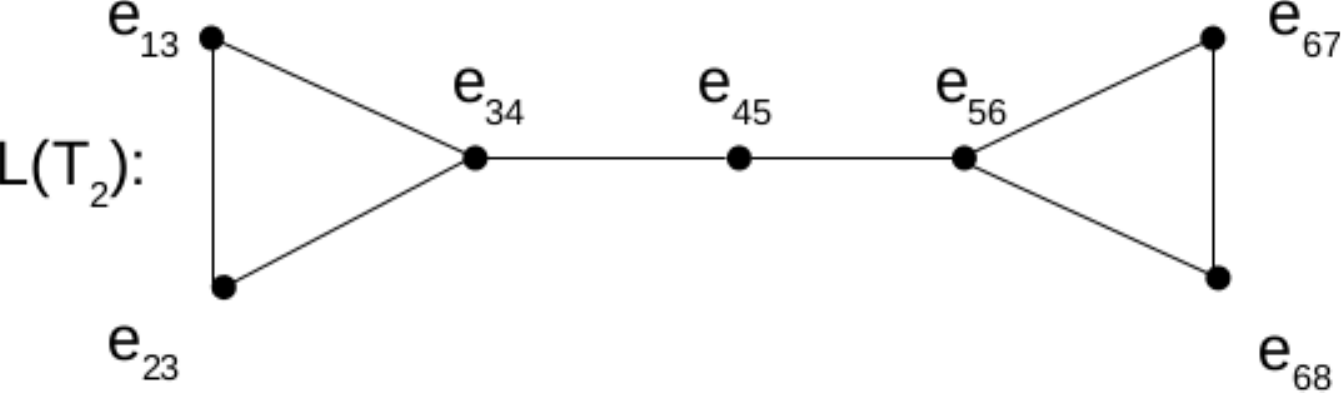}
		\end{figure}
	\end{center}
 Note that, in this case, any pair formed by an edge from the left triangle and one from the right one is a gap. The edges of the triangles are induced by the neighbors of the vertices $3$ and $6$. Hence, in general, we have $\displaystyle{{|\mathcal{N}(u)\setminus \{v\}|}\choose{2}}\cdot{{|\mathcal{N}(v)\setminus \{u\}|}\choose{2}}$. We kept the notation from the required formula, but, in this case $\mathcal{N}(v)\setminus \{u\}=\mathcal{N}(v)$ and $\mathcal{N}(u)\setminus \{v\}=\mathcal{N}(u)$.

One may also note in the figure that there are also two gaps, $\{e_{34},e_{45}\},\{e_{67},e_{68}\}$ and $\{e_{45},e_{56}\},\{e_{13},e_{23}\}$, but they come from the vertices $4,6$, and $3,5$ respectively which are at distance $2$, so they were discussed above.
 \end{proof}
The computation of the number of cycles in a graph is ofinterest in combinatorics \cite{Al}. The results obtained so far allow us to determine the number of induced cycles in $\overline{L(T)}$, where $T$ is a tree. We keep the notation from the above results.
\begin{Proposition}
	Let $T$ be a tree. Then $\overline{L(T)}$ has $$\sum\limits_{u,v,w\in V(T)}|N(u)\setminus (N(v)\cup N(w))|\cdot|N(v)\setminus (N(u)\cup N(w))|\cdot|N(w)\setminus (N(u)\cup N(v))|$$ cycles $C_3$, $c$ cycles $C_4$, and no cycle of length greater than or equal to $5$.
\end{Proposition}

\begin{proof} Since $L(T)$ is a chordal graph, then it is weakly chordal. In particular, $\overline{L(T)}$ does not contain any cycle of length greater than or equal to $5$. According to Proposition \ref{betti2-c}, $L(T)$ contains $c$ gaps, therefore in $\overline{L(T)}$ there are exactly $c$ cycles of length $4$. For computing the numbers of $C_3$ in $\overline{L(T)}$, one has to note that each such a cycle comes from non-adjacent vertices from three different maximal cliques, and each maximal clique is given by the neighbors of a cutpoint. Therefore, there are exactly   $$\sum\limits_{u,v,w\in V(T)}|N(u)\setminus (N(v)\cup N(w))|\cdot|N(v)\setminus (N(u)\cup N(w))|\cdot|N(w)\setminus (N(u)\cup N(v))|$$cycles of length three in $\overline{L(T)}$.
	\end{proof}
\section{Caterpillar graph and its line graph}
We consider now a particular class of trees, namely caterpillar trees and we pay attention on the projective dimension and the Krull dimension of $I(L(T))$.

We recall that \textit{a caterpillar graph} is a tree in which the removal of all pendant vertices results in a chordless path. The chordless path is called \textit{the backbone of the graph}. The edges from the backbone to the pendant vertices are called \textit{the hairs} of the caterpillar graph. 
Firstly, we compute the Castelnuovo--Mumford regularity of the edge ideal of the line graph of a caterpillar graph:
\begin{Proposition}
	Let $T$ be a caterpillar tree such that each vertex has degree one or less than or equal to four. Then $\reg\, R/I(L(T)))$ is equal to the number of cliques from the graph $L(T))$, or, equivalently, to the number of cutpoints of $T$.
\end{Proposition}
\begin{proof}One may note that an induced matching of $L(T)$ is obtained by taking the edges induced by two free vertices which are neighbors of the same cutpoint. More precisely, if we consider the following caterpillar graph,
	\begin{center}
		\begin{figure}[h]
			\includegraphics[height=2.5cm]{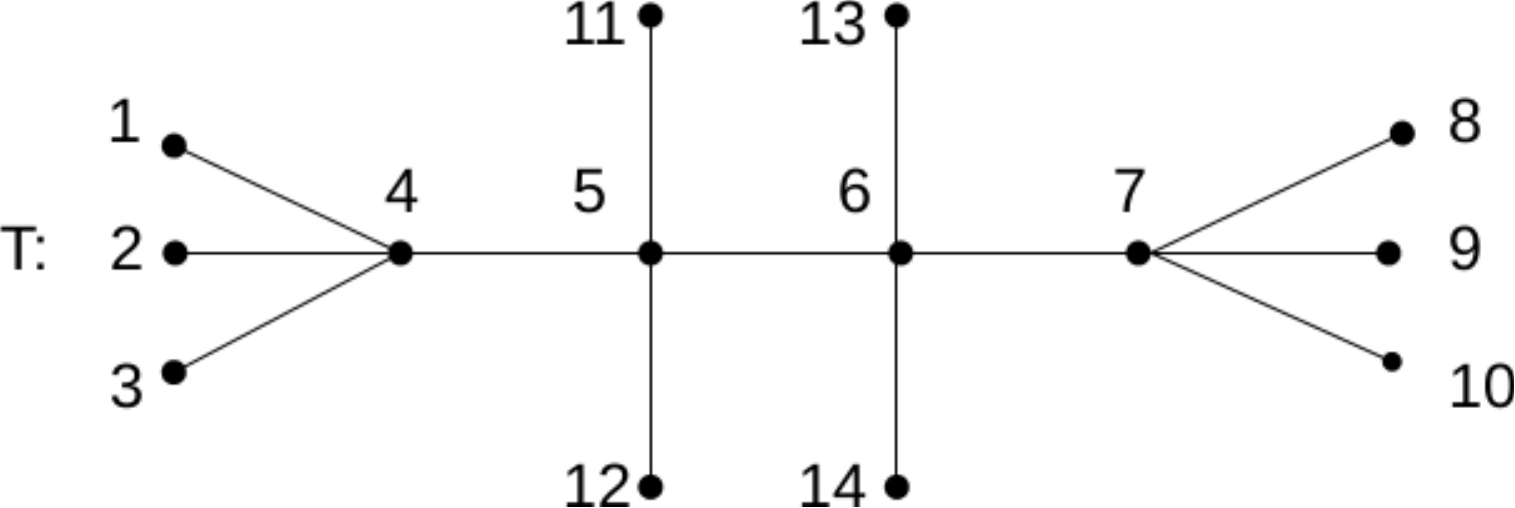}
		\end{figure}
	\end{center} 
	then the corresponding line graph is
	\begin{center}
		\begin{figure}[h]
			\includegraphics[height=3cm]{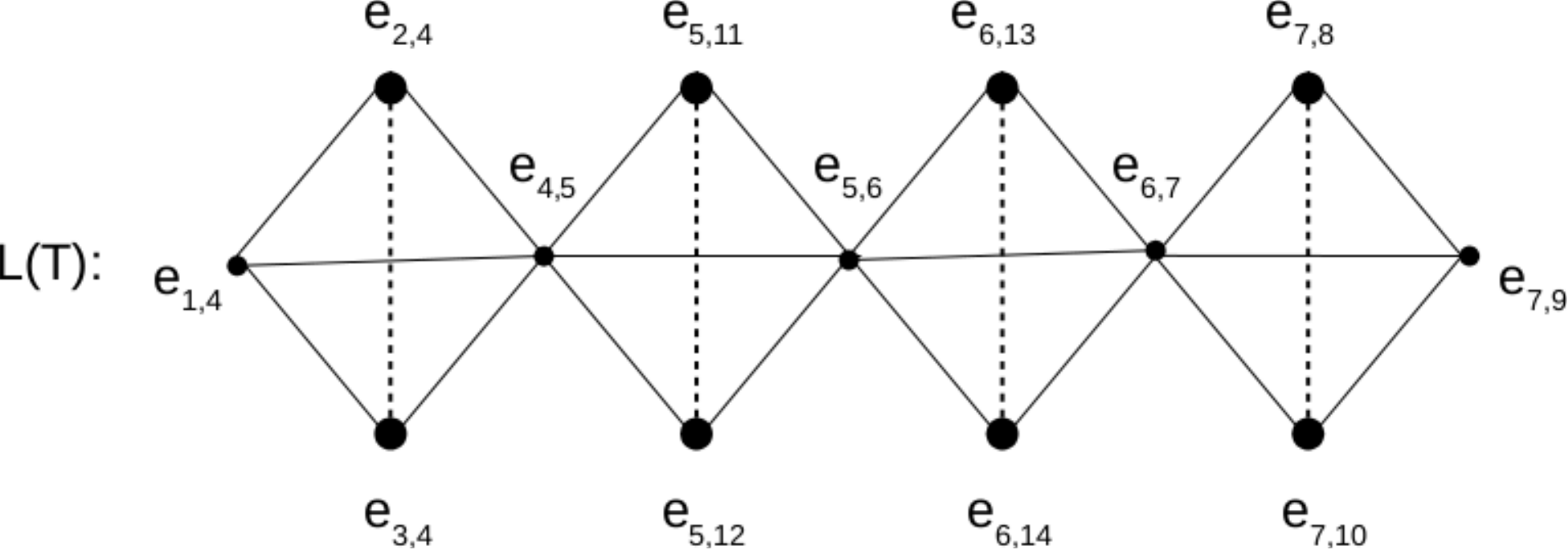}
		\end{figure}
	\end{center}
	and an induced matching is given by the dashed edges. Therefore, if $s$ is the number of cliques from $L(T)$, then $\reg R(I(L(T)))=\indmat L(T)\geq s$. According to Proposition \ref{regcliques}, one also obtain the converse inequality since one may consider as an induced independent set $J_0$ the set obtained by taking one vertex from each maximal clique from $L(T)$, that is by considering only one pendant edge for each cutpoint. Due to the restriction on the degree of cutpoints, the remaining cliques are of size at least $3$. 
\end{proof}
In order to compute the projective dimension, we will apply the results developped in \cite{Ki}. We follow \cite{Ki} in order to fix the notations.

A graph $B$ is called \textit{a bouquet} if $B$ is a star graph with $V(B)=\{w,z_1,\ldots,z_t\}$, $t\geq 1$, and the set of edges $E(B)=\{\{w,z_i\}:1\leq i\leq t\}$. The vertex $w$ is called \textit{the root of} $B$, the vertices $z_1,\ldots,z_t$ are called \textit{the flowers of} $B$ and the edges of the graph $B$ are called \textit{stems}. We denote by $F(B)$ the set of flowers of $B$. Let $G$ be a graph with the vertex set $V(G)$, $E(G)$ be its set of edges, and $\mathcal{B}=\{B_1,\ldots,B_r\}$ a set of bouquets of $G$. Then \[\mathcal{F}(\mathcal{B})=\{z\in V(G):z\mbox{ is a flower in some bouquet from }\mathcal{B}\},\]
\[\mathcal{R}(\mathcal{B})=\{w\in V(G): \mbox{ is a root in some bouquet from }\mathcal{B}\}.\]

A set $\mathcal{B}$ of bouquets of $G$ is called \textit{semi-strongly disjoint} if $V(B_i)\cap V (B_j) = \emptyset$ for all $i\neq j$ and any two vertices belonging to $\mathcal{R}(\mathcal{B})$ are not adjacent in $G$.

Let $d'_G:=\max\{|\mathcal{F}(\mathcal{B})|: \mathcal{B} \mbox{ is a semi-strongly disjoint set of bouquets of } G\}$.

For the case of chordal graphs, the projective dimension of the edge ideal can be computed in terms of semi-strongly disjoint sets. More precisely:
\begin{Theorem}\cite[Theorem 5.1]{K}\label{pdd'}
	Let $G$ be a chordal graph. Then $$\projdim S/I(G)=d'_G.$$
\end{Theorem} 

\begin{Proposition}\label{pdcat}
	Let $T$ be a caterpillar tree on $n$ vertices and $L(T)$ its line graph. Let $v_1,\ldots,v_r$ be the cutpoints from the backbone. Then $$\projdim(R/I(L(T)))=n-1-\left[\frac{r+1}{2}\right].$$ In particular, $$\depth R/I(L(T))=\left[\frac{r+1}{2}\right].$$
\end{Proposition}
\begin{proof} We use Theorem \ref{pdd'} in order to compute the projective dimension of $I(L(T))$. We split the proof in two cases:
	
	\textit{Case 1:} Let's assume that $r$ is even. We construct a set of bouquets $\mathcal{B}=\{B_1,\ldots,B_N\}$ with the set of roots
	\[\mathcal{R}(\mathcal{B})=\{e_{v_1v_2},e_{v_3v_4},\ldots,e_{v_{r-1}v_{r}}\}\] and the set of flowers \[\mathcal{F}(\mathcal{B})=\{\mathcal{F}(B_1),\ldots,\mathcal{F}(B_N)\},\] where $N=\left[\frac{r}{2}\right]=\left[\frac{r+1}{2}\right]$ and \[\mathcal{F}(B_1)=\mathcal{N}_{L(T)}\left(e_{v_1v_2}\right),\]
	\[\mathcal{F}(B_2)=\mathcal{N}_{L(T)}\left(e_{v_3v_4}\right)\setminus\{e_{v_2v_3}\},\]\[\ldots\]\[\mathcal{F}(B_i)=\mathcal{N}_{L(T)}\left(e_{v_{2i+1}v_{2i+2}}\right)\setminus\{e_{v_{2i}v_{2i+1}}\}\]\[\ldots\]
	Note that $$\mathcal{R}(\mathcal{B})\cup\mathcal{F}(\mathcal{B})=V(L(T)).$$ Therefore we get that $|\mathcal{R}(\mathcal{B})|=\left[\frac{r+1}{2}\right]$ and $$|\mathcal{F}(\mathcal{B})|=|V(L(T))|-\left[\frac{r+1}{2}\right]=n-1-\left[\frac{r+1}{2}\right].$$ We used here the fact that $|V(L(T))|=|E(T)|=n-1$ since $T$ is a tree on $n$ vertices. Therefore, $$d'_{L(T)}\geq n-1-\left[\frac{r+1}{2}\right].$$

	\textit{Case 2:} We assume now that $r$ is odd. We construct a set of bouquets $\mathcal{B}=\{B_1,\ldots,B_N\}$ with the set of roots
	\[\mathcal{R}(\mathcal{B})=\{e_{v_1v_2},e_{v_3v_4},\ldots,e_{v_{r}v_{r+1}}\}\] and the set of flowers \[\mathcal{F}(\mathcal{B})=\{\mathcal{F}(B_1),\ldots,\mathcal{F}(B_N)\},\] where $N=\left[\frac{r+1}{2}\right]$ and \[\mathcal{F}(B_1)=\mathcal{N}_{L(T)}\left(e_{v_1v_2}\right),\]
	\[\mathcal{F}(B_2)=\mathcal{N}_{L(T)}\left(e_{v_3v_4}\right)\setminus\{e_{v_2v_3}\},\]\[\ldots\]\[\mathcal{F}(B_i)=\mathcal{N}_{L(T)}\left(e_{v_{2i+1}v_{2i+2}}\right)\setminus\{e_{v_{2i}v_{2i+1}}\}\]\[\ldots\]
	Note that $$\mathcal{R}(\mathcal{B})\cup\mathcal{F}(\mathcal{B})=V(L(T)).$$ Therefore we get that $|\mathcal{R}(\mathcal{B})|=\left[\frac{r+1}{2}\right]$ and $$|\mathcal{F}(\mathcal{B})|=|V(L(T))|-\left[\frac{r+1}{2}\right]=n-1-\left[\frac{r+1}{2}\right].$$  As before, we obtain that $$d'_{L(T)}\geq n-1-\left[\frac{r+1}{2}\right].$$
	 One may easy note that, in both cases, the two sets of bouquets that we constructed contain the maximal number of flowers since a larger set of flowers will be given by considering less cutpoints. But due to the restriction of disjoint set of flowers, this will lead to a smaller set (in the set we considered, there are involved all the vertices which come from the free vertices and all the connecting edges). Thus, by Theorem \ref{pdd'}, $$\projdim R/I(L(T)) =d'_{L(T)}= n-1-\left[\frac{r+1}{2}\right]$$and $$\depth R/I(L(T))=\left[\frac{r+1}{2}\right].$$
\end{proof}
The following result allows us to determine the size of the largest vertex cover of the line graph of a caterpillar graph.
\begin{Proposition}\cite{FT,HHZ,Ki,MV}\label{pdbight} Let $G$ be a chordal graph. Then $$\projdim S/I(G)=\bight I(G).$$In particular, if $I(G)$ is unmixed, then $S/I(G)$ is Cohen--Macaulay.
\end{Proposition}
\begin{Proposition} Let $T$ be a caterpillar tree on $n$ vertices and $L(T)$ its line graph. Let $v_1,\ldots,v_r$ be the cutpoints of $T$. Then the largest maximal vertex cover is of size $n-1-\left[\frac{r+1}{2}\right]$.
\end{Proposition}
\begin{proof} The proof follows easily since $L(T)$ is a chordal graph and, by Proposition \ref{pdbight}, $\projdim R/I(L(T))=\bight (I(L(T)))$. One has to note that  $\bight (I(L(T)))$ gives the size of the largest maximal vertex cover of $L(T)$.
\end{proof}
In \cite{OV}, the set of maximal independent sets in caterpillar graphs is studied. We may determine the maximal size of a maximal independent set in the additional assumption that each cutpoint has degree at least $3$:

\begin{Proposition}
	Let $T$ be a caterpillar graph such that the degree of any cutpoint is at least three. Then $\dim R/I(L(T))=s$, where $s$ is the number of cutpoints in $T$ or the number of maximal cliques in $L(T)$.
\end{Proposition}
\begin{proof}The proof is straightforward by Proposition \ref{dimindep}, since $L(T)$ is formed by cliques of size at least $3$ and a maximal independent set can be obtained by taking one vertex from each clique. This is also the maximal size of a maximal independent set. The statement follows. 
	\end{proof}
\begin{Corollary}
	Let $T$ be a caterpillar graph with $n$ vertices such that the degree of any cutpoint is at least three. Then $\height I(L(T))=n-s$, where $s$ is the number of cutpoints of $T$. In particular, the minimal size of a minimal vertex cover is $n-s$.
\end{Corollary}
\begin{proof} The proof follows easily since $\dim S/I(L(T))=n-\height I(L(T))$ and any minimal prime ideal of $I(L(T))$ gives a minimal vertex cover of $L(T)$ (see \cite{MV,Vi} for more details.)
	\end{proof}
\section{Further comments and remarks}

The study of line graphs from commutative algebra point of view is not a new topic. As Eliahou and Villareal showed \cite{EV}, their connections to the minimal graded free resolutions make them a very interesting topic to study. It is nice to see whether properties like to be Cohen--Macaulay or Gorenstein are preserved by the line graph, or in which cases these properties are preserved 

 One may easily note that for cycles, they do not bring new information since $L(C_n)=C_n$ for any $n\geq3$. Therefore, one can consider the class of unicyclic graphs. One may try to characterize the edge ideals of the line graphs with a linear resolution.

\end{document}